\documentclass[12pt,reqno]{amsart}

\usepackage{times}
\usepackage[T1]{fontenc}
\usepackage{mathrsfs}
\usepackage{latexsym}
\usepackage[dvips]{graphics}
\usepackage{epsfig}
\usepackage{amsmath,amsfonts,amsthm,amssymb,amscd}
\input amssym.def
\input amssym.tex
\usepackage{color}
\usepackage{hyperref}
\usepackage{color}
\usepackage{breakurl}
\usepackage{comment}
\newcommand{\bburl}[1]{\textcolor{blue}{\url{#1}}}
\usepackage{mathtools}

\newcommand{\be}{\begin{equation}}
\newcommand{\ee}{\end{equation}}
\newcommand{\bea}{\begin{eqnarray}}
\newcommand{\eea}{\end{eqnarray}}

\newtheorem{thm}{Theorem}[section]
\newtheorem{conj}[thm]{Conjecture}
\newtheorem{cor}[thm]{Corollary}
\newtheorem{lem}[thm]{Lemma}

\newtheorem{defi}[thm]{Definition}

\newtheorem{rek}[thm]{Remark}

\newcommand{\Z}{\ensuremath{\mathbb{Z}}}

\newcommand{\N}{\mathbb{N}}

\numberwithin{equation}{section}

\providecommand{\floor}[1]{\lfloor #1 \rfloor}
\newcommand{\ph}{g}

\begin{document}

\title{When Sets Can and Cannot Have Sum-Dominant Subsets}

\author{H\`ung Vi\d{\^e}t Chu}
\email{\textcolor{blue}{\href{mailto:chuh19@mail.wlu.edu}{chuh19@mail.wlu.edu}}}
\address{Department of Mathematics, Washington and Lee University, Lexington, VA 24450}

\author{Nathan McNew}
\email{\textcolor{blue}{\href{mailto:nmcnew@towson.edu}{nmcnew@towson.edu}}}
\address{Department of Mathematics, Towson University, Towson, MD 21252}

\author{Steven J. Miller}
\email{\textcolor{blue}{\href{mailto:sjm1@williams.edu}{sjm1@williams.edu}},  \textcolor{blue}{\href{Steven.Miller.MC.96@aya.yale.edu}{Steven.Miller.MC.96@aya.yale.edu}}}
\address{Department of Mathematics and Statistics, Williams College, Williamstown, MA 01267}

\author{Victor Xu}
\email{\textcolor{blue}{\href{mailto:vzx@andrew.cmu.edu}{vzx@andrew.cmu.edu}}}
\address{Department of Mathematics, Carnegie Mellon University, Pittsburgh, PA 15213}

\author{Sean Zhang}
\email{\textcolor{blue}{\href{mailto:xiaoronz@andrew.cmu.edu}{xiaoronz@andrew.cmu.edu}}}
\address{Department of Mathematics, Carnegie Mellon University, Pittsburgh, PA 15213}



\date{\today}

\thanks{The first named author was supported by Washington \& Lee University, and the third named author was partially supported by NSF grants DMS1265673 and DMS1561945. We thank the students from the Math 21-499 Spring '16 research class at Carnegie Mellon and the participants from CANT 2016 for many helpful conversations, Angel Kumchev for comments on an earlier draft, and the referee for many suggestions which improved and clarified the exposition, as well as suggestions on good topics for future research.}

\begin{abstract} A finite set of integers $A$ is a sum-dominant (also called an More Sums Than Differences or MSTD) set if $|A+A| > |A-A|$. While almost all subsets of $\{0, \dots, n\}$ are not sum-dominant, interestingly a small positive percentage are. We explore sufficient conditions on infinite sets of positive integers such that there are either no sum-dominant subsets, at most finitely many sum-dominant subsets, or infinitely many sum-dominant subsets. In particular, we prove no subset of the Fibonacci numbers is a sum-dominant set, establish conditions such that solutions to a recurrence relation have only finitely many sum-dominant subsets, and show there are infinitely many sum-dominant subsets of the primes.
\end{abstract}

\maketitle

\tableofcontents

\section{Introduction}

For any finite set of natural numbers $A \subset \N$, we define the sumset \be A+A \ \coloneqq  \  \{a + a' : a, a' \in A\}\ee and the difference set \be A-A \ \coloneqq \  \{a - a' : a, a' \in A\};\ee $A$ is sum-dominant (also called a More Sums Than Differences or MSTD set) if $|A+A| > |A-A|$ (if the two cardinalities are equal it is called balanced, and otherwise difference-dominant). As addition is commutative and subtraction is not, it was natural to conjecture that sum-dominant sets are rare. Conway gave the first example of such a set, $\{0, 2, 3, 4, 7, 11, 12, 14\}$, and this is the smallest such set. Later authors constructed infinite families, culminating in the work of Martin and O'Bryant, which proved a small positive percentage of subsets of $\{0, \dots, n\}$ are sum-dominant as $n\to\infty$, and Zhao, who estimated this percentage at around $4.5 \cdot 10^{-4}$. See \cite{FP, He, HM, Ma, MO, Na1, Na2, Na3, Ru1, Ru2, Zh3} for general overviews, examples, constructions, bounds on percentages and some generalizations, \cite{MOS, MPR, MS, Zh1} for some explicit constructions of infinite families of sum-dominant sets, and \cite{DKMMW, DKMMWW, MV, Zh2} for some extensions to other settings.

Much of the above work looks at finite subsets of the natural numbers, or equivalently subsets of $\{0, 1, \dots, n\}$ as $n\to\infty$. We investigate the effect of restricting the initial set on the existence of sum-dominant subsets. In particular, given an infinite set $A = \{a_k\}_{=1}^\infty$, when does $A$ have no sum-dominant subsets, only finitely many sum-dominant subsets, or infinitely many sum-dominant subsets? \emph{We assume throughout the rest of the paper that every such sequence $A$ is strictly increasing and non-negative.}

Our first result shows that if the sequence grows sufficiently rapidly and there are no `small' subsets which are sum-dominant, then there are no sum-dominant subsets.

\begin{thm}\label{thm:gen} Let $A = \{a_k\}_{k=1}^\infty$ be a strictly increasing sequence of non-negative numbers. If there exists a positive integer $r$ such that
\begin{enumerate}
	\item $a_k > a_{k-1} + a_{k-r}$ for all $k\ge r+1$, and
	\item $A$ does not contain any sum-dominant set $S$ with $|S| \le 2r-1$,
\end{enumerate}
then $A$ contains no sum-dominant set.
\end{thm}

We prove this in \S\ref{sec:noMSTD}. As the smallest sum-dominant set has 8 elements (see \cite{He}), the second condition is trivially true if $r \le 4$. In particular, we immediately obtain the following interesting result.

\begin{cor} No subset of the Fibonacci numbers $\{0, 1, 2, 3, 5, 8, \dots\}$ is a sum-dominant set. \end{cor}

The proof is trivial, and follows by taking $r=3$ and noting \be F_k \ = \ F_{k-1} + F_{k-2} \ > \ F_{k-1} + F_{k-3} \ee for $k \ge 4$.

After defining a class of subsets we present a partial result on when there are at most finitely many sum-dominant subsets.

\begin{defi}[Special Sum-Dominant Set] For a sum-dominant set $S$, we call $S$ a special sum-dominant set if $|S+S| - |S-S| \ge |S|$.  \end{defi}

We prove sum-dominant sets exist in \S\ref{sec:specialsumdom}. Note if $S$ is a special sum-dominant set then if $S' = S \cup \{x\}$ for any sufficiently large $x$ then $S'$ is also a sum-dominant set. We have the following result about a sequence having at most finitely many sum-dominant sets (see \S\ref{sec:finitelymany} for the proof).

\begin{thm}\label{thm:finite} Let $A = \{a_k\}_{k=1}^\infty$ be a strictly increasing sequence of non-negative numbers. If there exists a positive integer $s$ such that the sequence $\{a_k\}$ satisfies
\begin{enumerate}
	\item $a_k > a_{k-1} + a_{k-3}$ for all $k\ge s$, and
	\item $\{a_1, \dots, a_{4s+6}\}$ has no special sum-dominant subsets,
\end{enumerate}
then $A$ contains at most finitely many sum-dominant sets.
\end{thm}

The above results concern situations where there are not many sum-dominant sets; we end with an example of the opposite behavior.

\begin{thm}\label{thm:prime} There are infinitely many sum-dominant subsets of the primes. \end{thm}

We will see later that this result follows immediately from the Green-Tao Theorem \cite{GT}, which asserts that the primes contain arbitrarily long progressions. We also give a conditional proof in \S\ref{sec:primes}. There we assume the Hardy-Littlewood conjecture (see Conjecture \ref{conj:HL}) holds. The advantage of such an approach is that we have an explicit formula for the number of the needed prime tuples up to $x$, which gives a sense of how many such solutions exist in a given window.

\section{Subsets with no sum-dominant sets}\label{sec:noMSTD}

We prove Theorem \ref{thm:gen}, establishing a sufficient condition to ensure the non-existence of sum-dominant subsets.

\begin{proof}[Proof of Theorem \ref{thm:gen}]
Let $S=\{s_1,s_2,\dots, s_k\} = \{a_{g(1)}, a_{g(2)}, \dots, a_{g(k)}\}$ be a finite subset of $A$, where $g : \Z^+ \to \Z^+$ is an increasing function. We show that $S$ is not a sum-dominant set by strong induction on $g(k)$.

We proceed by induction. We show that if $A$ has no sum-dominant subsets of size $k$, then it has no sum-dominant subsets of size $k+1$; as any sum-dominant set has only finitely many elements, this completes the proof. 

For the Basis Step, we know (see \cite{He}) that all sum-dominant sets have at least 8 elements, so any subset $S$ of $A$ with exactly $k$ elements is not a sum-dominant set if $k \le 7$; in particular, $S$ is not a sum-dominant set if $g(k) \le 7$. Thus we may assume for $g(k) \ge 8$ that all $S'$ of the form $\{s_1,\dots, s_{k-1}\}$ with $s_{k-1} < a_{g(k)}$ are not sum-dominant sets. The proof is completed by showing \begin{equation} S \ = \ S' \cup \{a_{g(k)}\} \ =\  \{s_{1},\dots, s_{k-1}, a_{g(k)}\}\end{equation} is not sum-dominant sets for any $a_{g(k)}$.

We now turn to the Inductive Step. We know that $S'$ is not a sum-dominant set by the inductive assumption. Also, if $k \le 2r-1$ then $|S| \le 2r-1$ and $S$ is not a sum-dominant set by the second assumption of the theorem. If $k \ge 2r$, consider the number of new sums and differences obtained by adding $a_{g(k)}$. As we have at most $k$ new sums, the proof is completed by showing there are at least $k$ new differences.

Since $k \ge 2r$, we have $k - \floor{\frac{k+1}{2}} \ge r$. Let $t = \floor{\frac{k+1}{2}}$. Then $t \le k - r$, which implies $s_{t} \le s_{k-r}$. The largest difference in absolute value between elements in $S$ is $s_{k-1}-s_1$; we now show that we have added at least $k+1$ distinct differences greater than $s_{k-1}-s_1$ in absolute value, which will complete the proof. We have
\begin{align}
a_{g(k)} - s_{t} &\ \ge \ a_{g(k)} - s_{k-r} \ = \ a_{g(k)} - a_{g(k-r)}  \nonumber\\
                   &\ \ge \ a_{g(k)} - a_{g(k)-r} \nonumber\\
                   &\  >  \ a_{g(k)-1}-a_1 & \text{(by the first assumption on $\{a_n\}$)} \nonumber\\
                   &\ \ge \ s_{k-1} - a_{1} \ \ge \  s_{k-1} - s_{1}.
\end{align}
Since $a_{g(k)} - s_{t} \ge s_{k-1}-s_1$, we know that $$a_{g(k)} - s_{t},\ \dots,\ a_{g(k)} - s_{2},\ a_{g(k)} - s_{1}$$ are $t$ differences greater than the greatest difference in $S'$. As we could subtract in the opposite order, $S$ contains at least \begin{equation} 2t\ =\ 2\left\lfloor\frac{k+1}{2}\right\rfloor\ \ge \ k\end{equation} new differences. Thus $S+S$ has at most $k$ more sums than $S'+S'$  but $S-S$ has at least $k$ more differences compared to $S'-S'$. Since $S'$ is not a sum-dominant set, we see that $S$ is not a sum-dominant set.
\end{proof}

\begin{rek} We thank the referee for the following alternative proof. Given any infinite increasing sequence $\{a_{g(i)}\}$ that is a subset of a set $A$ satisfying $a_k > a_{k_1} + a_{k-r}$ for all $k > r$,  let $S_k = \{a_{g(1)}, \dots, a_{g(k)}\}$ and $\Delta_k = |S_k - S_k| - |S_k + S_k|$. Similar arguments as above show that $\{\Delta_k\}$ is increasing for $k \ge 2r$. \end{rek}

We immediately obtain the following.

\begin{cor}\label{cor:abs} Let $A = \{a_k\}_{k=1}^\infty$ be a strictly increasing sequence of non-negative numbers.
If $a_k > a_{k-1} + a_{k-4}$ for all $k\ge 5$, then $A$ contains no sum-dominant subsets.
\end{cor}

\begin{proof}
From \cite{He} we know that all sum-dominant sets have at least 8 elements. When $r = 4$ the second condition of Theorem \ref{thm:gen} holds, completing the proof.
\end{proof}

For another example, we consider shifted geometric progressions.

\begin{cor}\label{thm:geo}
Let $A = \{a_k\}_{k=1}^\infty$ with $a_k = c \rho^k + d$ for all $k\ge 1$, where $0 \neq c \in \N$, $d\in\N$, and $1 < \rho \in \N$. Then $A$ contains no sum-dominant subsets.
\end{cor}

\begin{proof} Without loss of generality we may shift and assume $d=0$ and $c=1$; the result now follows immediately from simple algebra. \end{proof}

\begin{rek} Note that if $\rho$ is an integer greater than the positive root of $x^4 - x^3 - 1$ (the characteristic polynomial associated to $a_k = a_{k-1} + a_{k-4}$ from Theorem \ref{thm:finite}, which is approximately 1.3803) then the above corollary holds for $\{c \rho^k + d\}$. 
\end{rek}


\section{Subsets with Finitely Many sum-dominant Sets}\label{sec:finitelymany}

We start with some properties of special sum-dominant sets, and then prove Theorem \ref{thm:finite}. The arguments are similar to those used in proving Theorem \ref{thm:gen}. \emph{In this section, in particular in all the statements of the lemmas, we assume the conditions of Theorem \ref{thm:finite} hold.} Thus $A = \{a_k\}_{k=1}^\infty$ and there is an integer $s$ such that the sequence $\{a_k\}$ satisfies
\begin{enumerate}
	\item $a_k > a_{k-1} + a_{k-3}$ for all $k\ge s$, and
	\item $\{a_1, \dots, a_{4s+6}\}$ has no special sum-dominant subsets.
\end{enumerate}

\subsection{Special Sum-Dominant Sets}\label{sec:specialsumdom}

Recall a sum-dominant set $S$ is special if $|S+S| - |S-S| \ge |S|$. For any $x \ge \sum_{a\in S} a$, adding $x$ creates $|S|+1$ new sums and $2|S|$ new differences. Let $S^* = S\cup \{x\}$. Then \begin{equation} |S^* + S^*| - |S^* - S^*|\ \ge\ |S| + (|S|+1) - 2|S|\ =\ 1,\end{equation} and $S^*$ is also a sum-dominant set. Hence, from one special sum-dominant set $S \subset \{a_n\}_{n=1}^\infty =: A$, we can generate infinitely many sum-dominant sets by adding any large integer in $A$. We immediately obtain the following converse.

\begin{lem}\label{lem:appendingelementtonotspecial}  If a set $S$ is not a special sum-dominant set, then $|S+S| - |S-S| < |S|$, and by adding any large $x \ge \sum_{a\in S} a$, $S\cup \{x\}$ has at least as many differences as sums. Thus only finitely many sum-dominant sets can be generated by appending one integer from $A$ to a non-special sum-dominant set $S$. \end{lem}

Note that special sum-dominant sets exist. We use the base expansion method (see \cite{He}), which states that given a set $A$, for all $m$ sufficiently large if \be A_t \ = \ \left\{\sum_{i=1}^t a_i m^{i-1}: a_i \in A\right\} \ee then \be |A_t \pm A_t| \ = \ |A \pm A|^t; \ee the reason is that for $m$ large the various elements are clustered with different pairs of clusters yielding well-separated sums.  To construct the desired special sum-dominant set, consider the smallest sum-dominant set $S = \{0, 2, 3, 4$, $7$, $11$, $12, 14\}$. Using the method of base expansion, taking $m = 10^{2017}$ we obtain $S_3$ containing $|S_3| = 8^3 = 512$ elements such that $|S_3 + S_3| = |S + S|^3 = 26^3 = 17576$ and $|S_3 - S_3| = |S - S|^3 = 25^3 = 15625$. Then $|S_3 + S_3| - |S_3 - S_3| > |S_3|$.

\subsection{Finitely Many Sum-Dominant Sets on a Sequence}

If a sequence $A = \{a_n\}_{n=1}^\infty$ contains a special sum-dominant set $S$, then we can get infinitely many sum-dominant subsets on the sequence just by adding sufficiently large elements of $A$ to $S$. Therefore for a sequence $A$ to have at most finitely many sum-dominant subsets, it is necessary that it has no special sum-dominant sets. Using the result from the previous subsection, we can prove Theorem \ref{thm:finite}.

We establish some notation before turning to the proof in the next subsection. We can write $A$ as the union of $A_1 = \{a_1, \dots, a_{s-1}\}$ and $A_2 = \{a_{s}, a_{s+1}, \dots \}$. By Corollary \ref{cor:abs}, we know that $A_2$ contains no sum-dominant sets. Thus any sum-dominant set must contain some elements from $A_1$.


We prove a lemma about $A_2$.

\begin{lem}\label{lem:add}
Let $S' = \{s_1, \dots, s_{k-1}\}$ be a subset of $A$ containing at least 3 elements $a_{r_1}, a_{r_2}, a_{r_3}$ in $A_2$, with $r_3 > r_2 > r_1$. Consider the index $\ph(k) > r_3$, and let $S = S' \cup \{a_{\ph(k)}\}$. Then either $S$ is not a sum-dominant set, or $S$ satisfies $|S-S| - |S+S| > |S' - S'| - |S' + S'|$. Thus the excess of sums to differences from $S$ is \emph{less than} the excess from $S'$.
\end{lem}

\begin{proof}
We follow a similar argument as in Theorem \ref{thm:gen}.

If $k \le 7$, then $S$ is not a sum-dominant set.

If $k \ge 8$, then $k - \floor{\frac{k+3}{2}} \ge 3$. Let $t = \floor{\frac{k+2}{2}}$. Then $t \le k - 3$, and $s_{t} \le s_{k-3}$, and
\begin{align}
a_{\ph(k)} - s_{t} &\ \ge\ a_{\ph(k)} - s_{k-3} = a_{\ph(k)} - a_{\ph(k-3)}  \nonumber\\
                   &\ \ge\ a_{\ph(k)} - a_{\ph(k)-3} \nonumber\\
                   & \ >\ a_{\ph(k)-1} = a_{\ph(k)-1} - a_{1} & \text{(by assumption on $a$)} \nonumber\\
                   &\ \ge\ s_{k-1} - a_{1} \ge s_{k-1} - s_{1}.
\end{align}
In the set $S'$, the greatest difference is $s_{k-1}-s_1$. Since $a_{\ph(k)} - s_{t} \ge s_{k-1}-s_1$, we know that $a_{\ph(k)} - s_{t}, \dots, a_{\ph(k)} - s_{2}, a_{\ph(k)} - s_{1}$ are all differences greater than the greatest difference in $S'$.

By a similar argument, $s_{t} - a_{\ph(k)}, \dots, s_{2} - a_{\ph(k)}, s_{1} - a_{\ph(k)}$ are all differences smaller than the smallest difference in $S'$.

So $S$ contains at least $2t = 2\floor{\frac{k+3}{2}} > 2\cdot \frac{k+1}{2} = k+1$ new differences compared to $S'$, and $S$ satisfies
\be
|S-S| - |S+S| \ >\ |S' - S'| - |S' + S'|,
\ee completing the proof.
\end{proof}


\subsection{Proof of Theorem \ref{thm:finite}}


Recall that we write $A = A_1 \cup A_2$ with $A_1 = \{a_1$, $\dots$, $a_{s-1}\}$, $A_2 = \{a_{s}$, $a_{s+1}$, $\dots \}$, and by Corollary \ref{cor:abs} $A_2$ contains no sum-dominant sets (thus any sum-dominant set must contain some elements from $A_1$). We first prove a series of useful results which imply the main theorem.

Our first result classifies the possible sum-dominant subsets of $A$. Since any such set must have at least one element of $A_1$ in it but not necessarily any elements of $A_2$, we use the subscript $n$ below to indicate how many elements of $A_2$ are in our sum-dominant set.

\begin{lem}[Classification of Sum-Dominant Subsets of $A$]\label{lem:classificationA} Notation as above, let $K_n$ be a sum-dominant subset of $A = A_1 \cup A_2$ with $n$ elements in $A_2$. Thus we may write $$K_n\ =\ S \cup \{a_{r_1}, \dots, a_{r_n}\}$$ for some $$S \ \subset \ A_1 \ = \ \{a_1, \dots, a_s\}, \ \ \ s \le r_1 < r_2 < \dots < r_n.$$ Set $$d\ =\ \max_{K_3} (|K_3 + K_3| - |K_3 - K_3|, 1).$$ Then $n \le d+3$. In other words, a sum-dominant subset of $A$ can have at most $d+3$ elements of $A_2$. \end{lem}

\begin{proof} Let $S_m$ be any subset of $A$ with $m$ elements of $A_2$. Lemma \ref{lem:add} tells us that for any $S_m$ with $m \ge 3$, when we add any new element $a_{r_{m+1}}$ to get $S_{m+1}$, either $S_{m+1}$ is not a sum-dominant set, or $$|S_{m+1}-S_{m+1}| - |S_{m+1}+S_{m+1}| \ \ge\ |S_m -   S_m| - |S_m + S_m| + 1.$$

For an $n > d+3$, assume there exists a sum-dominant set; if so, denote it by $K_n$. For $3 \le k \le n$, define $S_k$ as the set obtained by deleting the $(n-k)$ largest elements from $K_n$ (equivalently, keeping only the $k$ smallest elements from $K_n$ which are in $A_2$). We prove that each $S_k$ is sum-dominant, and then show that this forces $S_n$ not to be sum-dominant; this contradiction proves the theorem as $K_n = S_n$.

If $S_k$ is not a sum-dominant set for any $k \ge 3$, by Lemma \ref{lem:add} either $S_{k+1}$ is not a sum-dominant set, or $$|S_{k+1} - S_{k+1}| - |S_{k+1} + S_{k+1}|\ \ge \ |S_k - S_k| - |S_k + S_k| + 1\ \ge \ 0,$$ in which case $S_{k+1}$ is also not a sum-dominant set (because $S_k$ is not sum-dominant, the set $S_{k+1}$ generates at least as many differences as sums). As we are assuming $K_n$ (which is just $S_n$) is a sum-dominant set, we find $S_{n-1}$ is sum-dominant. Repeating the argument, we find that $S_{n-2}$ down to $S_3$ must also all be sum-dominant sets, and we have
\be
|S_n - S_n| - |S_n + S_n|\ \ge \ |S_3 - S_3| - |S_3 + S_3| + (n-3).
\ee
Since $S_3$ is one of the $K_3$'s (i.e., it is a sum-dominant subset of $A$ with exactly three elements of $A_2$), by the definition of $d$ the right hand side above is at least $n-3-d$. As we are assuming $n > d+3$ we see it is positive, and hence $S_n$ is not sum-dominant. As $S_n = K_n$ we see that $K_n$ is not a sum-dominant set, contradicting our assumption that there is a sum-dominant set $K_n$ with $n > d+3$, proving the theorem.
\end{proof}

%
%
%

\begin{lem}\label{lem:precursortothm} For $n \ge 0$ let $k_n$ denote the number of subsets $K_n \subset A$ which are sum-dominant and contain exactly $n$ elements from $A_2$. We write \begin{equation}\label{eq:writingKn} K_n \ = \ S \cup \{a_{r_1}, \dots, a_{r_n}\} \ \ \ {\rm with}\ S \subset A_1.\end{equation} Then
\begin{enumerate}
\item $k_n$ is finite for all $n \ge 0$, and
\item every $K_n$ is not a special sum-dominant set.
\end{enumerate}
\end{lem}

\begin{proof} We prove each part by induction. It is easier to do both claims simultaneously as we induct on $n$. We break the analysis into $n \in \{0, 1, 2, 3\}$ and $n \ge 4$. The proof for $n=0$ is immediate, while $n \in \{1,2,3\}$ follow by obtaining bounds on the indices permissible in a $K_n$, and then $n \ge 4$ follows by induction. We thus must check (1) and (2) for $n \le 3$. While the arguments for $n \le 3$ are all similar, it is convenient to handle each case differently so we can control the indices and use earlier results, in particular removing the largest element in $A_2$ yields a set which is not a special sum-dominant set. \\ \


\noindent \emph{Case $n=0$:} As $A_1$ is finite, it has finitely many subsets and thus $k_0$, which is the number of sum-dominant subsets of $A_1$, is finite (it is at most $2^{|A_1|}$). Further any $K_0$ is a subset of $$ A_1\ =\ \{a_1, \dots, a_{s-1}\},$$ which is a subset of \begin{equation}\label{eq:Aprime} A'\ = \ \{a_1, \dots, a_{4s+6}\}.\end{equation} As we have assumed $A'$ has no special sum-dominant set, no $K_0$ can be a special sum-dominant set. \\ \


\noindent \emph{Case $n=1$:} We start by obtaining upper bounds on $r_1$, the index of the smallest (and only) element in our set coming from $A_2$. Consider the index $4s$. We claim that
\begin{equation} \label{eq:s-1}
a_{4s} \ >\ \sum_{a \in A_1} a.
\end{equation}

This is because $|A_1| < s$ and $a_k > a_{k-1} + a_{k-3}$ for all $k\ge s$, and hence
\begin{align*} \label{eq:s-1}
\sum_{a \in A_1} a
& \ <\ s\cdot a_{s}  \nonumber\\
& \ <\  \frac{s}{2} \left(a_{s} + a_{s+2} \right)
  \ <\  \frac{s}{2} \cdot a_{s+3} \nonumber\\
& \ <\  \frac{s}{4} \left(a_{s+3} + a_{s+5} \right)
  \ <\  \frac{s}{4} \cdot a_{s+6} \dots  \nonumber\\
& \ <\  \frac{s}{2^{\lceil \log_2 s\rceil}} a_{s+3\lceil\log_2(s)\rceil} \nonumber\\
& \ <\  a_{s+3s}  \ =\ a_{4s}
\end{align*} (by doing the above $\lceil\log_2 s\rceil$ times we ensure that $s/2^{\lceil \log_2 s\rceil} < 1$, and since $s \ge 1$ we have $3s \ge 3 \lceil\log_2(s)\rceil$). Therefore for all $r_{1}$ sufficiently large,
\begin{equation} \label{eq:t-1}
a_{r_{1}}  \ >\ a_{4s} \ > \ \sum_{a \in A_1} a.
\end{equation}

Clearly there are only finitely many sum-dominant subsets $K_1$ with $r_1 \le 4s$; the analysis is completed by showing there are no sum-dominant sets with $r_1 > 4s$. Imagine there was a sum-dominant $K_1$ with $a_{r_1} > a_{4s}$. Then $K_1$ is the union of a set of elements $S=\{s_1, \dots, s_m\}$ in $A_1$ and $a_{r_1}$ in $A_2$. As $\sum_{s\in S} s  < a_{r_1}$, by Lemma \ref{lem:appendingelementtonotspecial} we find $K_1$ is not a sum-dominant set.

All that remains is to show none of the $K_1$ are special sum-dominant sets. This is immediate, as each sum-dominant $K_1$ is a subset of $\{a_1, \dots, a_{4s}\}$, which is a subset of $A'$ (defined in \eqref{eq:Aprime}). As we have assumed $A'$ has no special sum-dominant set, no $K_1$ can be a special sum-dominant set. \\ \

\noindent \emph{Case $n=2$:} Consider the index $4s+3$. If $K_2$ is a sum-dominant set then it has two elements, $a_{r_1} < a_{r_2}$, that are in $A_2$. We show that if $r_2 \ge 4s+3$ then there can be no sum-dominant sets, and thus there are only finitely many $K_2$.

For all $r_{2} \ge 4s+3$,
\begin{equation} \label{eq:t-2}
a_{r_2} - a_{r_2 - 1} \ >\ a_{r_2-3} \ \ge\ a_{4s} \ >\ \sum_{a \in A_1} a.
\end{equation}

Assume there is a sum-dominant $K_2$ with $r_2 \ge 4s+3$. It contains some elements $S=\{s_1, \dots, s_m\}$ in $A_1$ and $a_{r_1}, a_{r_2}$ in $A_2$. We have $$a_{r_2} - a_{r_1}\ \ge\ a_{r_2} - a_{r_2-1} \ >\ \sum_{a\in S} a.$$ Therefore $a_{r_2} > \left( \sum_{a\in S} a\right) + a_{r_1}$, and $S\cup\{a_{r_1}\}$ is not a special sum-dominant set by the $n=1$ case\footnote{If $S' = S \cup \{a_{r_1}\}$ is sum-dominant then it is not special, while if it is not sum-dominant then clearly it is not a special sum-dominant set.}. Hence, by Lemma \ref{lem:appendingelementtonotspecial} we find $K_2 = (S\cup \{a_{r_1}\}) \cup \{a_{r_2}\}$ is not a sum-dominant set.

Finally, as $K_2$ is a subset of $\{a_1, \dots, a_{4s+1}\}$, which is a subset of $A'$, by assumption $K_2$ is not a special sum-dominant set. \\ \

\noindent \emph{Case $n=3$:} Let $K_3$ be a sum-dominant set with three elements from $A_2$. We show that if $r_3 \ge 4s+6$ then there are no such $K_3$; as there are only finitely many sum-dominant sets with $r_3 < 4s+6$, this completes the counting proof in this case.

Consider the index $4s+6$. For all $r_{3} \ge 4s+6$,
\begin{equation} \label{eq:t-3}
a_{r_3-3} - a_{r_3-4} \ >\ a_{r_3-6} \ \ge\ a_{4s} \ >\ \sum_{a \in A_1} a.
\end{equation}

Consider any $K_3$ with $r_3 \ge 4s+6$. We write $K_3$ as $S \cup \{a_{r_1}, a_{r_2}, a_{r_3}\}$ and $S \subset A_1$. If $|S| < 5$, we know that $|K_3| < 8$, and $K_3$ is not a sum-dominant set as such a set  has at least 8 elements. We can therefore assume that $|S| \ge 5$. We have two cases.\\ \

\noindent \emph{Subcase 1: $r_2 \le r_3-3$:} Thus $$a_{r_3} - a_{r_2} - a_{r_1}\ \ge\ a_{r_3} - a_{r_3-3} - a_{r_3-4}\ \ge\ a_{r_3-1} - a_{r_3-4}\ \ge\ a_{r_3 -2} > a_{r_3 -6} \ >\ \sum_{a \in S} a.$$ As $S\cup\{a_{r_1}, a_{r_2}\}$ is not a special sum-dominant set by the $n=2$ case\footnote{As before, if it is sum-dominant it is not special, while if it is not sum-dominant it cannot be sum-dominant special; thus we have the needed inequalities concerning the sizes of the sets.}, adding $a_{r_3}$ with $$a_{r_3}\ >\ \left( \sum_{s\in S} s\right) + a_{r_1} + a_{r_2}$$ creates a non-sum-dominant set by Lemma \ref{lem:appendingelementtonotspecial}. \\ \

\noindent \emph{Subcase 2: $r_2 > r_3 - 3$:} Using \eqref{eq:t-3} we find $$a_{r_3} - a_{r_2}\ \ge\ a_{r_3} - a_{r_3-1}\ >\ \sum_{a\in S} a$$ and $$a_{r_2} - a_{r_1} \ > \ a_{r_3-2} - a_{r_3-3}\ >\ \sum_{a\in S} a.$$ Therefore the differences between $a_{r_1}$, $a_{r_2}$, $a_{r_3}$ are large relative to the sum of the elements in $S$, and our new sums and new differences are well-separated from the old sums and differences. Explicitly, $K_3 + K_3$ consists of $S + S$, $a_{r_1} + S$, $a_{r_2} + S$, $a_{r_3} + S$, plus at most 6 more elements (from the sums of the $a_r$'s), while $K_3 - K_3$ consists of $S-S$, $\pm(a_{r_1} - S)$, $\pm(a_{r_2} - S)$, $\pm(a_{r_3} - S)$, plus possibly some differences from the differences of the $a_r$'s.

As $S$ is not a special sum-dominant set, we know $|S+S| - |S-S| < |S|$ (if $S$ is not sum-dominant the claim holds trivially, while if it is sum-dominant it holds because $S$ is not special). Thus for $K_3$ to be sum-dominant, we must have
\begin{eqnarray} 0 & \ < \ & |K_3 + K_3| - |K_3 - K_3| \nonumber\\ & \le & \left(|S+S| + 3|S| + 6\right) - \left(|S-S| + 6|S|\right) \nonumber\\ & < & 6 -2|S|; \nonumber
\end{eqnarray} as $|S| \ge 5$ this is impossible, and thus $K_3$ cannot be sum-dominant.

%
%
%
%

Finally, as again $K_3$ is a subset of $A' = \{a_1, \dots, a_{4s+6}\}$, no $K_3$ is a special sum-dominant set.\\ \

\noindent \emph{Case $n\ge 4$ (inductive step):} We proceed by induction. We may assume that $k_n$ is finite for some $n\ge 3$, and must show that $k_{n+1}$ is finite. By the earlier cases we know there is an integer $t_n$ such that if $K_n$ is a sum-dominant subset of $A$ with exactly $n$ elements of $A_2$, then the largest index $r_n$ of an $a_i \in K_n$ is less than $t_n$.

We claim that if $K_{n+1}$ is a sum-dominant subset of $A$ then each index is less than $t_{n+1}$, where
%
%
$t_{n+1}$ is the smallest index such that if $r_{n+1} \ge t_{n+1}$ then
\begin{equation} \label{eq:t}
a_{r_{n+1}} \ >\ \sum_{i < r_{n}} a_i.
\end{equation}

We write $$K_{n+1} \ = \ S \cup \{a_{r_1}, \dots, a_{r_n}, a_{r_{n+1}}\}, \ \ S \subset A_1, \ \ \{a_{r_1}, \dots, a_{r_n}\} \subset A_2.$$ We show that if $r_{n+1} \ge t_{n+1}$ then $K_{n+1}$ is not sum-dominant. Let $S_n = K_{n+1} \setminus \{a_{r_{n+1}}\}$. We have two cases.

\begin{itemize}
\item If $r_n < t_{n}$, then by the inductive hypothesis $S_n$ is not a special sum-dominant set. So adding $a_{r_{n+1}} > \sum_{x\in S_{n}} x$ to $S_n$ gives a non-sum-dominant set by Lemma \ref{lem:appendingelementtonotspecial}.

\item If $r_n \ge t_{n}$, then by the inductive hypothesis $S_n$ is not a sum-dominant set. So $|S_n - S_n| - |S_n + S_n| \ge 0$. Since $n \ge 3$, we can apply Lemma \ref{lem:add}, and either $K_{n+1} = S_n \cup \{a_{r_{n+1}}\}$ is not a sum-dominant set, or $$|K_{n+1} - K_{n+1}| - |K_{n+1} + K_{n+1}|\ >\ |K_n - K_n| - |K_n + K_n| > 0,$$ in which case $S_{n+1}$ is still not a sum-dominant set.

\end{itemize}

We conclude that for all sum-dominant sets $S_{n+1}$, we must have $r_{n+1} < t_{n+1}$. So $k_{n+1}$ is finite.

Consider any sum-dominant set $K_{n+1} = S_n \cup \{a_{r_{n+1}}\}$. Applying lemma \ref{lem:add} again, we have $|K_{n+1} - K_{n+1}| - |K_{n+1} + K_{n+1}| > |S_n - S_n| - |S_n + S_n|$. We know, from inductive hypothesis, that $S_n$ is not a special sum-dominant set. Therefore all possible $K_{n+1}$ are not special sum-dominant sets.

By induction, $k_n$ is finite for all $n\ge 0$, and all $K_n$ are not special sum-dominant sets.
\end{proof}


\begin{proof}[Proof of Theorem \ref{thm:finite}]
By Lemma \ref{lem:classificationA} every sum-dominant subset of $A$ is of the form $K_0$, $K_1$, $K_2$, $\dots$, $K_{d+3}$ where the $K_n$ are as in \eqref{eq:writingKn}. By Lemma \ref{lem:precursortothm} there are only finitely many sets of the form $K_n$ for $n \le d+3$, and thus there are only finitely many sum-dominant subsets of $A$.
\end{proof}

\section{Sum-Dominant subsets of  the prime numbers}\label{sec:primes}

We now investigate sum-dominant subsets of the primes. While Theorem \ref{thm:prime} follows immediately from the Green-Tao theorem, we first conditionally prove there are infinitely many sum-dominant subsets of the primes as this argument gives a better sense of what the `truth' should be (i.e., how far we must go before we find sum-dominant subsets).

\subsection{Admissible Prime Tuples and Prime Constellations}

We first consider the idea of prime $m$-tuples. A prime $m$-tuple $(b_1, b_2, \dots, b_m)$ represents a pattern of differences between prime numbers. An integer $n$ matches this pattern if $(b_1 + n, b_2 + n, \dots, b_m + n)$ are all primes.

A prime $m$-tuple $(b_1, b_2, \dots, b_m)$ is called admissible if for all integers $k \ge 2$, $\{b_1, b_2, \dots, b_m\}$ does not cover all values modulo $k$. If a prime $m$-tuple is not admissible, whenever $n > k$ then at least one of $b_1 + n, b_2 + n, \dots, b_m + n$ is divisible by $k$ and greater than $k$, so this cannot be an $m$-tuple of prime numbers (in this case the only $n$ which can lead to an $m$-tuple of primes are $n \le k$, and there are only finitely many of these).


It is conjectured in \cite{HL} that all admissible $m$-tuples are matched by infinitely many integers.

\begin{conj}[Hardy-Littlewood \cite{HL}]\label{conj:HL}
Let $b_1, b_2, \dots, b_m$ be $m$ distinct integers, $v_p(b) = v(p; b_1, b_2, \dots , b_m)$ the number of distinct residues of $b_1, b_2, \dots b_m$  to the modulus $p$, and $P(x$; $b_1$, $b_2$, $\dots$, $b_m)$ the number of integers $1 \le n \le x$ such that every element in $\{n + b_1, n + b_2, \dots , n + b_m\}$ is prime. Assume $(b_1, b_2, \dots, b_m)$ is admissible (thus $v_p(b) \neq p$ for all $p$). Then
\begin{equation}\label{eq:HLforPx}
P(x) \ \sim \ \mathfrak{S}(b_1, b_2, \dots, b_m) \int_2^x \frac{du}{(\log u)^m}
\end{equation}
when $x \to \infty$, where
$$
\mathfrak{S}(b_1, b_2, \dots, b_m)\ =\ \prod_{p \ge 2} \left(\left( \frac{p}{p-1}\right)^{m-1} \frac{p-v_p(b)}{p-1} \right) \ \neq \ 0.
$$
\end{conj}
As $(b_1, b_2, \cdots, b_m)$ is an admissible $m$-tuple, $v(p; b_1, b_2, \dots, b_m)$ is never equal to $p$ and equals $m$ for $p> \max\{|b_i-b_j|\}$. The product $\mathfrak{S}(b_1, b_2, \dots, b_m)$ thus converges to a positive number as each factor is non-zero and is $1 + O_m(1/p^2)$. Therefore this conjecture implies that every admissible $m$-tuple is matched by infinitely many integers.

\subsection{Infinitude of sum-dominant subsets of the primes}

We now show the Hardy-Littlewood conjecture implies there are infinitely many subsets of the primes which are sum-dominant sets.

\begin{thm}\label{thm:HLprimes} If the Hardy-Littlewood conjecture holds for all admissible $m$-tuples then the primes have infinitely many sum-dominant subsets. \end{thm}

\begin{proof}
Consider the smallest sum-dominant set $S=\{0, 2, 3, 4, 7, 11, 12, 14\}$. We know that $\{p, p+2s, p+3s, p+4s, p+7s, p+11s, p+12s, p+14s\}$ is a sum-dominant set for all positive integers $p, s$. Set $s=30$ and let $T = (0, 60, 90, 120, 210, 330, 360, 420)$. We deduce that if there are infinitely many $n$ such that $n+T = (n, n+60, n+90, n+120, n+210, n+330, n+360, n+420)$ is an 8-tuple of prime numbers, then there are infinitely many sum-dominant sets of prime numbers.

We check that $T$ is an admissible prime 8-tuple. When $m>8$, the eight numbers in $T$ clearly don't cover all values modulo $m$. When $m \le 8$, one sees by straightforward computation that $T$ does not cover all values modulo $m$.

By Conjecture \ref{conj:HL}, there are infinitely many integers $p$ such that every element of $\{p, p+60, p+90, p+120, p+210, p+330, p+360, p+420\}$ is prime. These are all sum-dominant sets, so there are infinitely many sum-dominant sets on primes.
\end{proof}

Of course, all we need is that the Hardy-Littlewood conjecture holds for one admissible $m$-tuple which has a sum-dominant subset. We may take $p=19$, which gives an explicit sum-dominant subset of the primes: $\{19, 79, 109, 139, 229, 349, 379, 439\}$ (a natural question is which sum-dominant subset of the primes has the smallest diameter). If one wishes, one can use the conjecture to get some lower bounds on the number of sum-dominant subsets of the primes at most $x$. The proof of Theorem \ref{thm:prime} follows similarly.

\begin{proof}[Proof of Theorem \ref{thm:prime}] By the Green-Tao theorem, the primes contain arbitrarily long arithmetic progressions. Thus for each $N \ge 14$ there are infinitely many pairs $(p,d)$ such that \begin{equation} \{p, p + d, p + 2d, \dots, p + Nd\} \end{equation} are all prime. We can then take subsets as in the proof of Theorem \ref{thm:HLprimes}.
\end{proof}


\section{Future Work}

We list some natural topics for further research.

\begin{itemize}

\item Can the conditions in Theorem \ref{thm:gen} or \ref{thm:finite} be weakened?

\item What is the smallest special sum-dominant set by diameter, and by cardinality? 

\item What is the smallest, in terms of its largest element, set of primes that is sum-dominant?

\end{itemize}


\bigskip

\end{document}